\documentclass[conference]{IEEEtran}

\ifCLASSINFOpdf
\else
\fi
\usepackage{url}

\usepackage[]{fancyhdr} %
\newcommand{\changefont}{\fontsize{8}{8}\selectfont}
\IEEEoverridecommandlockouts

\usepackage{anysize}
\usepackage[usenames, dvipsnames]{color}
\marginsize{1.2in}{0.9in}{0.9in}{1.5in}

\usepackage{amsfonts, amsmath, amssymb,amsthm,comment}  
\usepackage{times,enumerate}
\usepackage{nccmath}
		{\end{pmatrix}\end{medsize}}%
\usepackage{dsfont,bbm}
\usepackage[nottoc,numbib]{tocbibind}
\usepackage{rotating}
\usepackage{lscape}
\usepackage{multicol}

\usepackage{thmtools, thm-restate}
\usepackage[hyperindex=true,
			bookmarks=true,
            pdftitle={}, pdfauthor={},
            plainpages=false,
            pdfpagelabels,
            hyperfootnotes=false]{hyperref}
\hypersetup{
            	colorlinks,
            	linkcolor=blue,
            	citecolor=OliveGreen,
            	filecolor=magenta,      
            	urlcolor=cyan,
            }

\usepackage{scrextend}
\usepackage{graphicx}
\usepackage{color}
\DeclareGraphicsExtensions{.pdf,.png,.jpg}

\usepackage{float}
\usepackage{algorithm2e}

\makeatletter
\renewcommand{\fnum@algocf}{\AlCapSty{\AlCapFnt\nobreakspace\algorithmcfname}}%
\makeatother
\SetAlgorithmName{CC-ADMM}{}{}

\newtheorem{theorem}{Theorem}[section]

\newtheorem{proposition}[theorem]{Proposition}

\newtheorem*{question*}{Question}

\theoremstyle{definition}

\newtheorem{remark}[theorem]{Remark}
\newtheorem{example}[theorem]{Example}

\newtheorem{assumption}[theorem]{Assumption}

\newcommand{\set}[1]{\left\{#1\right\}}
\newcommand{\mc}{\mathcal}
\newcommand{\mbb}{\mathbb}
\newcommand{\mbf}{\mathbf}
\newcommand{\mds}{\mathds}
\newcommand{\mf}{\mathfrak}
\newcommand{\bs}{\boldsymbol}

\newcommand{\norm}[1]{\left\lVert#1\right\rVert}
\newcommand{\abs}[1]{\lvert#1\rvert}

\newcommand{\wt}{\widetilde}
\newcommand{\wh}{\widehat}
\newcommand{\ul}[1]{\underline{#1}}
\newcommand{\ol}[1]{\overline{#1}}

\usepackage{xcolor}

\newcommand{\ab}[1]{{\color{black}#1}}
\newcommand{\nn}[1]{\textcolor{black}{#1}}
\newcommand{\nnrev}[1]{\textcolor{black}{#1}}

\DeclareMathOperator{\transpose}{T}
\DeclareMathOperator{\complex}{i}

\DeclareMathOperator*{\argmin}{arg\,min}

\DeclareMathOperator{\st}{s.t.}
\usepackage{color}
\newcommand*{\mathcolor}{}
\def\mathcolor#1#{\mathcoloraux{#1}}
\newcommand*{\mathcoloraux}[3]{%
  \protect\leavevmode
  \begingroup
    \color#1{#2}#3%
  \endgroup
}
\newcommand{\sts}{\mathcolor{white}{\st}}

\usepackage[utf8]{inputenc}

\makeatletter
\newenvironment{varsubequations}[1]
 {%
  \addtocounter{equation}{-1}%
  \begin{subequations}
  \renewcommand{\theparentequation}{#1}%
  \def\@currentlabel{#1}%
 }
 {%
  \end{subequations}\ignorespacesafterend
 }
\makeatother

\begin{document}

%
\title{Communication-Censored-ADMM for Electric Vehicle Charging in Unbalanced Distribution Grids}

\author{
\IEEEauthorblockN{A. Bhardwaj*$^{\dagger}$}
\and
\IEEEauthorblockN{W. de Carvalho*}
\and
\IEEEauthorblockN{N. Nimalsiri*}
\and
\IEEEauthorblockN{E.~L.~Ratnam*}
\and
\IEEEauthorblockN{N. Rin*}
\thanks{\noindent *School of Engineering, The Australian National University, Canberra, Australia. \newline \indent $^{\dagger}$ Corresponding author Abhishek.Bhardwaj@anu.edu.au}
}


%





\maketitle
\thispagestyle{fancy}
\pagestyle{fancy}


\begin{abstract}

We propose an alternating direction method of multipliers (ADMM)-based algorithm for coordinating the charge and discharge of electric vehicles (EVs) to manage grid voltages while minimizing EV time-of-use energy costs. We prove that by including a Communication-Censored strategy, the algorithm maintains its solution integrity, while reducing peer-to-peer communications. By means of a case study on a representative unbalanced two node circuit, we demonstrate that our proposed Communication-Censored-ADMM \nnrev{(CC-ADMM)} EV charging strategy reduces peer-to-peer communications by up to $80\%$, compared to a benchmark ADMM approach.
\end{abstract}



%
\IEEEpeerreviewmaketitle

\section{Introduction}

The transportation sector is rapidly evolving to support the grid-integration of electric vehicles (EVs), unlocking financial and environmental benefits for EV customers \cite{macharis2007combining, shareef2016review}. EV charging however, potentially increases the duration and magnitude of grid congestion peaks when not managed carefully. Furthermore, during periods of grid congestion, EV charging  potentially fluctuates supply voltages beyond safe operating limits \cite{venegas2021active}. Consequently, a key challenge with EV proliferation is coordinating EV charging (and discharging) to ensure steady-state supply voltages across an electrical distribution grid remain within prescribed operating limits \cite{nimalsiri2021coordinated}.


Numerous optimization strategies have been proposed to coordinate EV (dis)charging to 
manage voltages throughout the distribution grid \ab{(\cite{DeHoog2014Optimal}, \cite{huo2020decentralized}, \cite{Liu2017Decentralized}), with recent works tackling the issue of model uncertainties (see \cite{wan2018model}, \cite{li2019constrained} for reinforcement learning approaches, and \cite{yuan2021towards} for perspectives on online optimization)}. 
Optimization-based EV coordination approaches can be split into two main classes - \emph{centralized} \nn{and} \emph{distributed} \footnote{Some authors in the literature (such as Khaki et al. in \cite{khaki2019hierarchical}) use the term distributed \nnrev{even} when including a coordinating agent. In this paper, distributed means fully decentralized --- i.e., the EVs only communicate amongst themselves and not with a coordinating agent.}. Comparatively, distributed approaches consider ways to parallelize and scale mathematical computations, allowing for larger populations of EVs to be coordinated across an electrical network \cite{molzahn2017survey}. Over the last decade, 
the \emph{Alternating Direction Method of Multipliers} (ADMM) \cite{boyd2011distributed} technique has provided a simple yet effective approach to distributed EV (dis)charging. 

The ADMM-based \nn{EV} coordination algorithms presented in \cite{zhou2020voltage} \nn{and} \cite{rahman2021continuous} are designed to avoid system voltages breaches. The authors in \cite{khaki2018hierarchical} propose to minimize grid load variations using EV (dis)charging with a hierarchical ADMM \nnrev{approach}. More recently\nnrev{,} Nimalsiri et al. \ab{propose in} \cite{nimalsiri2021distributed} an ADMM-based approach to minimize time-of-use EV charge operational costs while regulating unbalanced supply voltages to within operational limits. Yet none of the aforementioned approaches consider ways to reduce the communication cost associated with ADMM-based distributed algorithms. 

Specifically, the distributed algorithms in \cite{rahman2021continuous, nimalsiri2021distributed} 
assume a communication network where each EV must communicate at every time step, with all other EVs in its vicinity. For large scale networks, this communication overhead dwarfs the associated distributed computation costs. Consequently, a tradeoff exists for EV scalability, in the context of ease of computation and growing communication costs, which has been studied in different contexts, e.g., \cite{tsianos2012communication, nedic2018network, berahas2018balancing}. As such, we are motivated to reduce ADMM-based communications in addition to minimizing time-of-use EV charge costs, all while regulating unbalanced grid voltages to avoid quasi steady-state operation breaches. 



In this paper, we propose a 
distributed Communication-Censored-ADMM EV (dis)charging algorithm, which we refer to as \emph{\ref{alg1}}, to reduce the overall communication costs of coordinating EV (dis)charging, while regulating unbalanced grid voltages to avoid quasi steady-state operation breaches. Furthermore, we seek to minimize time-of-use EV charge costs, all while preserving the ADMM-based algorithm accuracy. Spec\nn{i}fically, we extend the ADMM-based algorithm presented in \cite{nimalsiri2021distributed} for coordinated EV (dis)charging in a number of ways. First, we incorporate EV-based inverter reactive power control for improved voltage regulation. Then, we integrate the communication censored-ADMM approach of \cite{liu2019communication} with our model to create CC-ADMM. Most importantly, we obtain theoretical convergence guarantees that the CC-ADMM converges to an optimal solution of the EV (dis)charge problem. By means of a case study, we compare communication costs and solution accuracy \nn{of CC-ADMM} against the benchmark Dis-Net-EVCD algorithm from 
\cite{nimalsiri2021distributed}, for two examples. 




The remainder of this paper is organized as follows. 
In Section \ref{S:setup} we introduce a residential EV system which is connected to an unbalanced distribution grid. In Section \ref{S:optimodel} we present the EV (dis)charging problem as a centralized optimization problem. Section \ref{S:comms} reformulates the centralized EV (dis)charging problem into the ADMM framework, and shows that the CC-ADMM algorithm converges to an optimal solution. Section \ref{S:numerics} includes numerical simulation results to demonstrate the potential of the communication censor\nn{ed} approach. 

\section*{Notation}
We use the symbols $\mbb{R}^{n}, \mbb{C}^{n}$ to denote the $n$-dimensional vector space of real and complex numbers, respectively. Throughout, a vector $\mbf{x}\in\mbb{C}^{n}$ represents a column vector. We denote by $\mbf{x}^{\transpose}\in\mbb{C}^{n}$ the \emph{transpose} of $\mbf{x}$, which is a row vector. For a vector $\mbf{x}\in\mbb{C}^{n}$ we write $\mbf{x}_{i}$ for the $i^{\text{th}}$ component of $\mbf{x}$. We write $\mbb{R}^{n}_{\geq0}$ for the non-negative orthant in  (the set of vectors which have only non-negative components). We write $\mds{1}_{n}\in\mbb{R}^{n}$ for the vector which has all components equal to 1 (we omit the subscript $n$ when the dimension is clear from the context). The inner product of vectors $\mbf{x}, \mbf{y}\in\mbb{R}^{n}$ is denoted by $\mbf{x}^{\transpose}\mbf{y}$ or $\left\langle \mbf{x}, \mbf{y} \right\rangle$. Given two vectors $\mbf{x}, \mbf{y}\in\mbb{R}^{n}$ the inequality $\mbf{x}\leq \mbf{y}$ is to be understood componentwise (i.e., $\mbf{x}_{i}\leq \mbf{y}_{i}$ for $i=1,\dotsc,n$). We denote by $\norm{\mbf{x}}$ the Euclidian norm of $\mbf{x}\in\mbb{R}^{n}$, which is $\sqrt{\mbf{x}^{\transpose}\mbf{x}}$.

The set of real and complex matrices are denoted by $\mbb{R}^{n\times n}$ and $\mbb{C}^{n\times n}$, respectively. Given a matrix $W\in\mbb{C}^{n\times n}$, we write $[W]_{i,j}$ for the entry in row $i$ and column $j$ of the matrix $W$. We denote the identity matrix in $\mbb{C}^{n\times n}$ by $I_{n}$ (we omit the subscript when the size is clear from the context). The zero vector in $\mbb{C}^{n}$ is denoted by $\mbf{0}_{n}$, and the zero matrix in $\mbb{C}^{n\times n}$ is denoted by $\mbf{0}_{n\times n}$, where we omit the subscripts when the distinction and sizes can be easily deduced by the context. Given matrices $A_{1}, \dotsc, A_{m}\in\mbb{C}^{n\times n}$, we denote the \emph{direct sum} as
$$
\bigoplus_{i=1}^{m}A_{i} = 
\begin{bmatrix}
    A_{1} & \mbf{0}_{n\times n} & \dotsc & \dotsc & \mbf{0}_{n\times n}\\
    \mbf{0}_{n\times n} & A_{2} & \mbf{0}_{n\times n} & \dotsc & \mbf{0}_{n\times n}\\
    \vdots & \ddots & \ddots & \ddots & \vdots\\
    \vdots & \ddots & \ddots & \ddots & \vdots\\
    \mbf{0}_{n\times n} & \dotsc & \dotsc & \dotsc & A_{m}
\end{bmatrix}.
$$

Given two sets $A, B$, we use the following standard notations
\begin{align*}
    A\times B & = \set{(a,b) \;\middle\vert\; a\in A, \ b\in B},  \\
    A\cup B & = \set{ c \;\middle\vert\; c\in A \text{ or } c\in B },\\
    A\cap B & = \set{ c \;\middle\vert\; c\in A \text{ and } c\in B },
\end{align*}
and we denote by $\abs{A}$ the cardinality of the set $A$. Further notation will be introduced as necessary.

\section{Problem Formulation}\label{S:setup}

\subsection{Residential EV Energy System}

Fig.~\ref{fig:convention} illustrates the topology of a grid-connected Residential EV Energy System for a single customer, consisting of \ab{a meter,} %
an EV and residential load situated behind the Point of Common Coupling (PCC). In what follows, we consider a set of $N$ customers $\mc{N}:=\set{1,\dotsc,n,\dotsc,N}$, where each customer $n$ charges or discharges $EV_{n}$. We assume that for each customer, $EV_{n}$ is enabled with both real and reactive power control by way of the EV battery inverter.

Let $\Delta$ be a time interval length \ab{\footnote{\ab{Emerging grid sensors such as micro-phasor measurement units (micro-PMUs) provide voltage and current phasor data at a high temporal resolution (see \cite{von2014micro, zhou2016abnormal}). We anticipate customer metering technology to follow a similar trajectory, reporting high temporal voltage and current data, which would also improve the applicability of our approach.}}}, and $\mc{T}:=\set{1,\dotsc,T}$ be a time horizon partition, so that the planning horizon is $[0,\Delta T]$. 

Let $p_{n}(t)\in\mbb{R}$ denote the \emph{real} power to or from  $EV_{n}$ over the time period \ab{$[\Delta(t-1), \Delta t]$} (in kW), and similarly, let $q_{n}(t)\in\mbb{R}$ denote the \emph{reactive} power to or from $EV_{n}$ over the time period \ab{$[\Delta(t-1), \Delta t]$} (in kVAR). 

\begin{figure}
    \centering
    \includegraphics[scale=0.38]{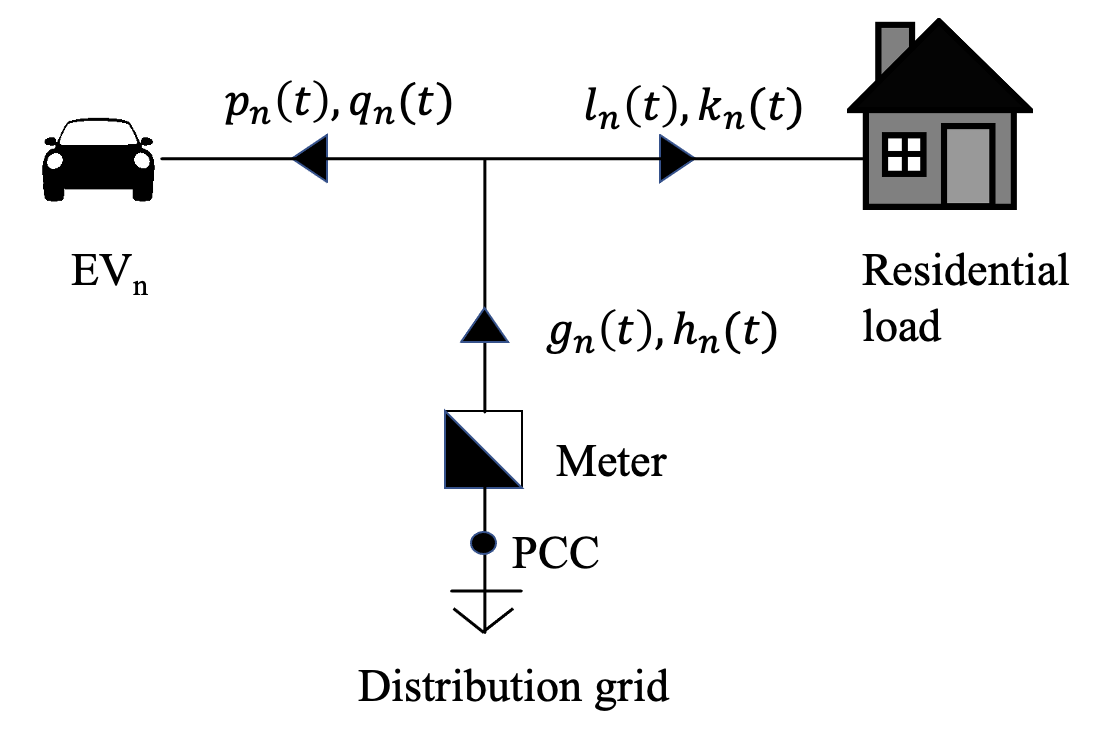}
    \caption{A Residential EV Energy System, where the arrows indicate the positive direction of power flow. For example, non-negative $p_{n}(t)$ is real power to $EV_{n}$, and negative $p_{n}(t)$ is real power from $EV_{n}$. Here, $g_{n}(t),h_{n}(t),l_{n}(t),k_{n}(t)$ are the real and reactive powers through the PCC, and the residential load, respectively.}
    \label{fig:convention}
\end{figure}



The \emph{power profile} of $EV_{n}$ over \ab{$[0,\Delta T]$} is defined as
\begin{equation*}
    \left(\mbf{p}_{n}, \mbf{q}_{n}\right)
    =
    \left(
    \begin{bmatrix}
    p_{n}(1)\\ 
    \vdots  \\
    p_{n}(T) 
    \end{bmatrix},
    \begin{bmatrix}
    q_{n}(1)\\ 
    \vdots  \\
    q_{n}(T) 
    \end{bmatrix}
    \right)
    \in\mbb{R}^{T}\times\mbb{R}^{T}.
\end{equation*}

Each $EV_{n}$ has associated with it the following set of design parameters; 
\emph{arrival time index} $a_{n}\in\mc{T}$, \emph{departure time index} $d_{n}\in\mc{T}$, \emph{battery capacity} $b_{n}$, \emph{inverter capacity} $\ol{s}_{n}$, \emph{initial state of charge (SoC)} $\hat{\sigma}_{n}$, \emph{target} SoC $\sigma^{*}_{n}$, \emph{minimum} SoC $\ul{\sigma}_{n}$, \emph{maximum} SoC $\ol{\sigma}_{n}$, \emph{maximum charge rate} $\ol{p}_{n}$, and \emph{minimum charge rate} $\ul{p}_{n}$. 
These design parameters are collected into the set
\begin{equation*}
    \Lambda_{n} := \set{
    \begin{gathered}
    a_{n}, d_{n}, b_{n}, \ol{s}_{n}, \hat{\sigma}_{n}, \sigma_{n}^{*}, 
    \ol{\sigma}_{n}, \ul{\sigma}_{n}, \ol{p}_{n}, \ul{p}_{n}
    \end{gathered}
    }.
\end{equation*}   
\begin{remark}
\ab{We do not include battery (dis)charge efficiencies in the set of design parameters for convenience}. Adding these parameters will convert our optimization-based problem formulation (which follows) into a more difficult mixed integer program. However, with careful consideration, the core ideas we present in this paper can also be applied in the mixed integer program setting.
\end{remark}
Let the apparent power to or from $EV_n$ be denoted by ${s}_{n}(t)$, defined (for all \ab{$\Delta t \in[0,\Delta T]$)} by
\begin{equation*}
    {s}_{n}(t)^{2} := p_{n}(t)^{2} + q_{n}(t)^{2}.
\end{equation*}
The apparent power through the inverter of $EV_n$ is constrained as $s_{n}(t) \leq \ol{s}_{n}$, where $\ol{s}_{n}$ denotes the inverter capacity. It follows that, for all \ab{$\Delta t \in[0,\Delta T]$},
\begin{equation}\label{eq:q_bound}
    p_{n}(t)^{2} + q_{n}(t)^{2} \leq \ol{s}_{n}^{2}.
\end{equation}

Each customer $n$ must specify a departure time and target SoC which is feasible within the departure time, i.e., it must hold that
\begin{equation*}
    \sigma^{*} \leq \hat{\sigma} + \ab{\Delta(d_{n}-a_{n})\ol{p}_{n}}. 
\end{equation*}
The battery SoC of $EV_{n}$ at time \ab{$\Delta t\in[0\Delta T]$} is denoted by $\sigma_{n}(t)$, which can be defined as
\begin{equation}\label{eq:SoC_instant}
    \sigma_{n}(t) := \hat{\sigma} + \ab{\Delta\sum_{j=1}^{t}p_{n}(j)},
\end{equation}
and the \emph{SoC profile} of $EV_{n}$ is then defined as
\begin{equation*}
    \bs{\sigma}_{n} := 
    \begin{bmatrix}
    \sigma_{n}(1)\\
    \vdots \\
    \sigma_{n}(t)\\
    \vdots \\
    \sigma_{n}(T)
    \end{bmatrix}
    \in\mbb{R}^{T}_{\geq0}.
\end{equation*}

\subsection{EV Battery Constraints and EV Operation Costs}

The battery of $EV_{n}$ is constrained to prevent over-(dis)charging, in the following way
\begin{equation}\label{eq:SoC_constraint}
    \ul{\sigma}_{n}\leq \sigma_{n}(t) \leq \ol{\sigma}_{n},
\end{equation}
which implies that the SoC profile must satisfy
\begin{equation*}
    \ul{\sigma}_{n}\mds{1} \leq \bs{\sigma}_{n} \leq \ol{\sigma}_{n}\mds{1}.
\end{equation*}

Let \ab{$\ul{\alpha}_{n} = (\ul{\sigma}_{n} - \hat{\sigma}_{n})/\Delta$, and $\ol{\alpha}_{n} = (\ol{\sigma}_{n} - \hat{\sigma}_{n})/\Delta$}. Then using \eqref{eq:SoC_instant} together with \eqref{eq:SoC_constraint} gives
\begin{equation}
    \ab{\ul{\alpha}_{n}}\mds{1}\leq M \mbf{p}_{n} \leq \ab{\ol{\alpha}_{n}}\mds{1},
\end{equation}
where 
\begin{equation*}
    M = \begin{bmatrix}
        1 & 0 & \dotso & \dotso & 0\\
        1 & 1 & 0 & \dotso & 0\\
        \vdots & \vdots & \dotso & \dotso & 0\\
        1 & 1 & \dotso & \dotso & 1
        \end{bmatrix}
        \in\mbb{R}^{T\times T}.
\end{equation*}
Due to the limited (dis)charging power of the battery, we have
\ab{
\begin{equation*}
    \ul{p}_{n}\leq p_{n}(t) \leq \ol{p}_{n},
\end{equation*}
}
which implies that the power profile vector $\mbf{p}_{n}$ is constrained as
\begin{equation}\label{eq:p_bound}
    \ul{p}_{n}\mds{1}\leq \mbf{p}_{n} \leq \ol{p}_{n}\mds{1}.
\end{equation}
To incorporate the charging demand for $EV_{n}$, we let \ab{$e_{n} = (\sigma^{*}_{n} - \hat{\sigma}_{n})/\Delta$} and enforce $\sigma_{n}(T) \geq \sigma_{n}^{*}$. Combining this charging demand with \eqref{eq:SoC_instant} gives
\begin{equation*}
\mds{1}^{\transpose}\mbf{p}_{n} \geq e_{n}.
\end{equation*}
To limit the duration of (dis)charging for $EV_{n}$ to the period $EV_{n}$ is grid connected ($[a_{n},d_{n}]$), we define the availability matrix $L_{n}$ as follows
\begin{equation*}
    \left[L_{n}\right]_{i,j} 
    := 
    \begin{cases}
    1, & \text{ if } i=j \text{ and } a_{n}<i\leq d_{n},\\
    0, & \text{ otherwise. }
    \end{cases}
\end{equation*}
Accordingly, we further constrain $\mbf{p}_{n}$ and $\mbf{q}_{n}$ by 
\begin{equation*}
    \left( I - L_{n} \right)\mbf{p}_{n} = \left( I - L_{n} \right)\mbf{q}_{n} = \mbf{0},
\end{equation*}
where $I\in\mbb{R}^{T\times T}$ is the identity matrix, and $\mbf{0}\in\mbb{R}^{T}$ is the zero vector. Let us define 
$$
    \mbf{F}_{n}:=\begin{bmatrix}
    \phantom{-}I & \mbf{0}_{T\times T}\\
    -I& \mbf{0}_{T\times T}\\
    \phantom{-}M& \mbf{0}_{T\times T}\\
    -M& \mbf{0}_{T\times T}\\
    \phantom{-}\mds{1}^{\transpose}& \mbf{0}_{T\times T}\\
    \left( I - L_{n} \right)& \mbf{0}_{T\times T}\\
    \left( L_{n} - I \right)& \mbf{0}_{T\times T}\\
    \mbf{0}_{T\times T} & \left( I - L_{n} \right)\\
    \mbf{0}_{T\times T} & \left(L_{n} - I\right)\\
    \end{bmatrix},
    \quad 
    \mbf{f}_{n}:=\begin{bmatrix}
    \phantom{-}\ul{p}_{n}\mds{1}\\
    -\ol{p}_{n}\mds{1}\\
    \phantom{-}\ul{\alpha}_{n}\mds{1}\\
    -\ol{\alpha}_{n}\mds{1}\\
    e_{n}\\
    \mbf{0}\\
    \mbf{0}\\
    \mbf{0}\\
    \mbf{0}
    \end{bmatrix}.
$$
We now combine all of the battery constraints and the apparent power constraints, to write the \emph{feasibility set} of power profile vectors
\begin{equation}\label{eq:feasible_set}
    \mc{F}_{n}
    := 
    \set{ (\mbf{p}_{n}, \mbf{q}_{n})
    \;\middle\vert\;
    \begin{gathered}
    \mbf{F}_{n}\begin{bmatrix}
        \mbf{p}_{n}\\
        \mbf{q}_{n}
    \end{bmatrix}\geq\mbf{f}_{n}, \\
    p_{n}(t)^{2}+q_{n}(t)^{2}\leq \ol{s}_{n}^{2} \quad \forall t\in\mc{T}.
    \end{gathered}
    }
\end{equation}
The following proposition can be easily proven. 
\begin{proposition}\label{P:compact}
$\mc{F}_{n}$ is compact and convex.
\end{proposition}


We assume that each customer is compensated for delivering power to the grid at the same rate they are billed for consuming power from the grid. For our time horizon $\mc{T}$, we let the \emph{price profile} be denoted by
\begin{equation*}
    \bs{\eta}
    =
    \begin{bmatrix}
    \eta(1)\\
    \vdots\\
    \eta(t)\\
    \vdots\\
    \eta(T)
    \end{bmatrix}
    \in\mbb{R}^{T}_{\geq0},
\end{equation*}
where $\eta(t)$ is the price of electricity (in \textdollar/kWh) over the time period \ab{$[\Delta(t-1), \Delta t]$}. The operational cost for $EV_{n}$ is then defined by
\begin{equation}\label{eq:cost_func}
    \Omega_{n}(\mbf{p}_{n}) :=\ab{\Delta\bs{\eta}^{\transpose}\mbf{p}_{n} +  \kappa_{n}\mbf{p}_{n}^{\transpose}\mbf{p}_{n}},
\end{equation}
where $\kappa_{n}$ is a (small) regularization parameter which reduces to occurrence of unnecessary EV battery charge-discharge cycles. Specifically, we include the second term in \eqref{eq:cost_func} to maintain a smoother battery profile and avoiding unnecessary charge-discharge cycles \cite{rivera2013alternating}, and in this way, we include the cost of battery degradation within our operational costs for $EV_{n}$.
Positive values of $\Omega_{n}(\mbf{p}_{n})$ represent financial cost\nn{s} whereas negative values of $\Omega_{n}(\mbf{p}_{n})$ \nn{represent} financial gains.



\subsection{Unbalanced Distribution Grid Model}

We consider a three phase, unbalanced, radial distribution feeder that includes two phase and single phase laterals, represented by a graph $\mc{G} = \set{\mc{V}_{0}, \mc{E}}$. The set of vertices $\mc{V}_{0}=\set{0,\dotsc,i,\dotsc,k}$ represent nodes along the feeder and the edges $\mc{E} \subseteq \mc{V}_{0}\times \mc{V}_{0}$ represent line segments along the feeder with cardinality $\abs{\mc{E}}=k$. 

Node $0$ is considered the root node (feeder head) which is taken as an infinite bus, decoupling the interaction between the feeder and the wider power grid. An edge $(ij)\in\mc{E}$ exists if there is \nn{a} line segment between \nn{node} $i$ and \nn{node} $j$, with node $i$ being closest to node $0$.

We assume that the graph $\mc{G}$ is simple with no repeated edges or self loops for any $i\in\mc{V}_{0}$. Let $\mc{V} = \mc{V}_{0}\backslash \set{0}$, and for each vertex $i\in\mc{V}$ let $\mc{E}^{i}\subseteq \mc{E}$ be the set of edges on the unique path from node $0$ to node $i$.

Each edge $(ij)\in\mc{E}$ is either a single phase, two phase, or three phase edge, and each vertex $i\in\mc{V}_{0}$ is either a single phase, two phase, or three phase node, where phases are labelled $\mf{a}, \mf{b}$ and $\mf{c}$. The set of all phases at node $i\in\mc{V}_{0}$ are denoted $\Phi^{i}$, e.g., for a three phase node we would have $\Phi^{i} = \set{\mf{a}, \mf{b}, \mf{c}}$. We assume node $0$ is a three phase node.

The set of phases along edge $(ij)\in\mc{E}$ is denoted as $\Phi^{ij}$, e.g., $\Phi^{ij} = \set{\mf{a}, \mf{b}, \mf{c}}$ for a three phase line segment. For any edge $(ij)\in\mc{E}$ we have 
\begin{equation*}
    \Phi^{ij} \subseteq \left(\Phi^{i}\cap\Phi^{j}\right).
\end{equation*}
Next, we denote phases as $\phi$ and $\psi$, where $\phi, \psi\in\set{\mf{a}, \mf{b}, \mf{c}}$. 

Let $\set{(i:\phi)}_{\phi\in\Phi^{i}}$ represent the tuple of phases at node $i\in\mc{V}$, e.g., if $i$ is a three phase node, then we have $\set{(i:\phi)}_{\phi\in\Phi^{i}} = \set{ (i:\mf{a}),(i:\mf{b}),(i:\mf{c}) }$. Now let
\begin{equation*}
    \mc{K} = \bigcup_{i\in\mc{V}} \set{(i:\phi)}_{\phi\in\Phi^{i}}
\end{equation*}
with $\abs{\mc{K}} = K$ giving the sum of the total number of phases across all nodes in $\mc{V}$. Each  element $(i:\phi)\in\mc{K}$ represents a \emph{supply point} which serves $N^{i:\phi}\geq0$ customers, such that 
\begin{equation}
    \sum_{i\in\mc{V}}\sum_{\phi\in\Phi^{i}} N^{i:\phi} = N.
\end{equation}

We label the customers by initializing the customer index at $n=1$ and increment the index through all customers at a supply point, before continuing the index at each following supply point, and across all supply points in $\mc{K}$ in ascending order, respectively.

To identify the location of each customer with reference to a supply point $(i:\phi)\in\mc{K}$, let $\Theta$ be a binary matrix with rows indexed by elements of $\mc{K}$, and columns indexed by elements of $\mc{N}$, and define the entries of $\Theta$ as
\begin{equation}\label{eq:connection_matrix}
    \left[\Theta\right]_{(i:\phi),n} \negmedspace
    =\negmedspace
    \begin{cases}
    1, & \text{if customer $n\in\mc{N}$ is}\\
       & \text{connected to supply}\\
       & \text{point $(i:\phi)\in\mc{K}$}\\
    0, & \text{otherwise.}
    \end{cases}
\end{equation}

Let $r^{ij,\phi\psi}$ and $x^{ij,\phi\psi}$ denote the resistance and reactance of line segment $(ij)\in\mc{E}$ respectively, where $\phi$ and $\psi$ are the phases along the line segment. Then $z^{ij,\phi\psi}:=r^{ij,\phi\psi}+\mf{i}x^{ij,\phi\psi}$ (where $\mf{i}^{2}=-1)$, so that each line segment has a complex impedance $\mbf{z}^{ij}$, e.g., for a three phase line segment
\begin{equation*}
    \mbf{z}^{ij} = 
    \begin{bmatrix}
    z^{ij,\mf{aa}} & z^{ij,\mf{ab}} & z^{ij,\mf{ac}}\\
    z^{ij,\mf{ba}} & z^{ij,\mf{bb}} & z^{ij,\mf{bc}}\\
    z^{ij,\mf{ca}} & z^{ij,\mf{cb}} & z^{ij,\mf{cc}}
    \end{bmatrix}
    \in\mbb{C}^{3\times3}.
\end{equation*}

Let $Z^{ij,\phi\psi}:=\sum_{(uv)\in \left(\mc{E}^{i}\cap\mc{E}^{j}\right)}z^{uv,\phi\psi}\in\mbb{C}$ denote the sum of the impedances along the unique path which is common to the paths from node $0$ to node $i$, and node $0$ to node $j$, such that both paths incorporate phases $\phi$ and $\psi$.

Let $\omega = e^{\frac{-2\pi\complex}{3}}$ and let $[[\phi]]$ denote phase $\phi$ as a numeral, where we set $[[\mf{a}]]=0, [[\mf{b}]]=1, [[\mf{c}]]=2$. Let $R\in\mbb{R}_{\geq0}^{K\times K}$ be a square matrix with rows and columns indexed by elements of $\mc{K}$, with the entries defined as
\begin{equation*}
    [R]_{(i:\phi),(j:\psi)} = 2\Re\set{(Z^{ij,\phi\psi})^{*}\cdot \omega^{[[\phi]]-[[\psi]]}}\in\mbb{R}
\end{equation*}
where $i,j\in\mc{V}, \phi\in\Phi^{i}$, and $\psi\in\Phi^{j}$. We similarly define the matrix $X$ as
\begin{equation*}
    [X]_{(i:\phi),(j:\psi)} = -2\Im\set{(Z^{ij,\phi\psi})^{*}\cdot \omega^{[[\phi]]-[[\psi]]}}\in\mbb{R}
\end{equation*}
where $i,j\in\mc{V}, \phi\in\Phi^{i}$, and $\psi\in\Phi^{j}$. $\Re$ and $\Im$ denote the real and imaginary parts respectively, and $^{*}$ is the complex conjugate.

Let $v^{i:\phi}(t)$ denote the squared voltage magnitude of phase $\phi\in\Phi^{i}$ at node $i\in\mc{V}$, at time $t$. For a node $i\in\mc{V}$ and time $t\in\mc{T}$, let 
\begin{equation*}
    \mbf{v}^{i}(t) := 
    \begin{bmatrix}
    \set{v^{i:\phi}(t)}_{\phi\in\Phi^{i}}
    \end{bmatrix}^{\transpose}
    \in\mbb{R}^{\abs{\Phi^{i}}}
\end{equation*}
for example, at a three phase node we have
\begin{equation*}
    \mbf{v}^{i}(t) = 
    \begin{bmatrix}
    v^{i:\mf{a}}(t)\\
    v^{i:\mf{b}}(t)\\
    v^{i:\mf{c}}(t)
    \end{bmatrix}.
\end{equation*}
Collecting across all nodes in $\mc{V}$ we define
\begin{equation*}
    \mbf{V}(t) = 
    \begin{bmatrix}
    \mbf{v}^{1}(t)\\
    \vdots\\
    \mbf{v}^{i}(t)\\
    \vdots\\
    \mbf{v}^{k}(t)
    \end{bmatrix}\in\mbb{R}^{K}
\end{equation*}
and let $v^{0}$ be the squared voltage magnitude of each phase at the root node, which we set to the nominal voltage. We also define $\mbf{V}^{0} = v^{0}\mds{1}\in\mbb{R}^{K}$.

Let ${p}^{i:\phi}(t)\in\mbb{R}$ denote the net real power (in \ab{kW}) to or from node $i\in\mc{V}$ on phase $\phi\in\Phi^{i}$ over the time period \ab{$[\Delta(t-1),\Delta t]$}. In particular, ${p}^{i:\phi}(t)$ is the net real power to or from $N^{i:\phi}$ customers connected to phase $\phi\in\Phi^{i}$ at node $i\in\mc{V}$. Given a node $i\in\mc{V}$ we let
\begin{equation*}
    \mf{p}^{i}(t) :=
    \begin{bmatrix}
    \set{ 
    p^{i:\phi}(t)
    }_{\phi\in\Phi^{i}}
    \end{bmatrix}^{\transpose}\in\mbb{R}^{\abs{\Phi^{i}}}
\end{equation*}
and collecting over all nodes in $\mc{V}$ we define
\begin{equation*}
    \mbf{P}(t):=
    \begin{bmatrix}
    \mf{p}^{1}(t)\\
    \vdots\\
    \mf{p}^{i}(t)\\
    \vdots\\
    \mf{p}^{k}(t)
    \end{bmatrix}
    \in\mbb{R}^{K}.
\end{equation*}
Similarly, let $q^{i:\phi}\ab{(t)}\in\mbb{R}$ denote the net reactive power to or from node $i\in\mc{V}$ on phase $\phi\in\Phi^{i}$ over the time period \ab{$[\Delta(t-1),\Delta t]$}. And analogously define
\begin{equation*}
    \mf{q}^{i}(t) := 
    \begin{bmatrix}
    \set{
    q^{i:\phi}(t)
    }_{\phi\in\Phi^{i}}
    \end{bmatrix}\in\mbb{R}^{\abs{\Phi^{i}}}
\end{equation*}
and
\begin{equation*}
    \mbf{Q}(t) = 
    \begin{bmatrix}
    \mf{q}^{1}(t)\\
    \vdots\\
    \mf{q}^{i}(t)\\
    \vdots\\
    \mf{q}^{k}(t)
    \end{bmatrix}\in\mbb{R}^{K}.
\end{equation*}

Since each supply point $(i:\phi)\in\mc{K}$ serves $N^{i:\phi}$ customers, $p^{i:\phi}(t)$ and $q^{i:\phi}(t)$ have contributions from the EV and non-EV load. Accordingly, let us denote by $\tilde{p}^{i:\phi}(t)$ and $\tilde{q}^{i:\phi}(t)$, respectively, the net real and reactive power to or from the non-EV loads at node $i\in\mc{V}$ on phase $\phi\in\Phi^{i}$ over the time period \ab{$[\Delta(t-1),\Delta t]$}. Likewise, we denote by $\hat{p}^{i:\phi}(t)$ and $\hat{q}^{i:\phi}(t)$, respectively, the real and reactive power to or from the EV loads at supply point $(i:\phi)\in\mc{K}$ over the time period \ab{$[\Delta(t-1),\Delta t]$}. 
Ignoring losses, we define
\begin{align*}
    p^{i:\phi}(t) & := \tilde{p}^{i:\phi}(t)+\hat{p}^{i:\phi}(t),\\
    q^{i:\phi}(t) & := \tilde{q}^{i:\phi}(t)+\hat{q}^{i:\phi}(t).
\end{align*}
Similarly, we decompose terms $\mf{p}^{i}(t)$, $\mf{q}^{i}(t)$, $\mbf{P}(t)$ and $\mbf{Q}(t)$, such that for all $t\in\mc{T}$, we define
\begin{align*}
    \mbf{P}(t) & := \wt{\mbf{P}}(t) + \wh{\mbf{P}}(t),\\
    \mbf{Q}(t) & := \wt{\mbf{Q}}(t) + \wh{\mbf{Q}}(t).
\end{align*}

\ab{

}

\subsubsection{Power Flow Equations}
Consider now the power flow equations from \cite{gan2014convex, arnold2016optimal}, which extend the LinDistFlow equations to the unbalanced setting by
\begin{equation*}
    \mbf{V}(t) = \mbf{V}^{0} - R\mbf{P}(t) - X\mbf{Q}(t).
\end{equation*}
Let 
\begin{equation*}
    \wt{\mbf{V}}(t):= \mbf{V}^{0} -R\wt{\mbf{P}}(t) - X\wt{\mbf{Q}}(t),
\end{equation*}
which defines the baseline voltage from the aggregate non-EV loads over the time period \ab{$[\Delta(t-1),\Delta t]$}. Then, our primary equation of interest becomes
\begin{equation}\label{eq:LinDistFlow}
    \mbf{V}(t) = \wt{\mbf{V}}(t) - R\wh{\mbf{P}}(t) - X\wh{\mbf{Q}}(t).
\end{equation}

We now reformulate \eqref{eq:LinDistFlow} in terms of the power transfer from each customer in $\mc{N}$ (instead of each supply point in $\mc{K})$. We define
\begin{equation*}
    \mc{P}(t):=
    \begin{bmatrix}
    p_{1}(t)\\
    \vdots\\
    p_{n}(t)\\
    \vdots\\
    p_{N}(t)
    \end{bmatrix}\in\mbb{R}^{N},
    \,
    \mc{Q}(t):=
    \begin{bmatrix}
    q_{1}(t)\\
    \vdots\\
    q_{n}(t)\\
    \vdots\\
    q_{N}(t)
    \end{bmatrix}\in\mbb{R}^{N}
\end{equation*}
as vectors of real and reactive power to or from all EVs in $\mc{N}$ over the time period \ab{$[\Delta(t-1),\Delta t]$} (where clearly, the $n^{\text{th}}$ element is the power to or from $EV_{n}$). Then, together with the matrix $\Theta$ \eqref{eq:connection_matrix}, we write
\begin{equation*}
    \wh{\mbf{P}}(t) = \Theta\mc{P}(t), \quad \wh{\mbf{Q}}(t) = \Theta\mc{Q}(t),
\end{equation*}
for any $t\in\mc{T}$. Now let $D:=-R\Theta\in\mbb{R}^{K\times N}$ and $E:=-X\Theta\in\mbb{R}^{K\times N}$, and write
\begin{align*}
    D & = 
    \begin{bmatrix}
    D_{1} & \dotso & D_{n} & \dotso & D_{N}
    \end{bmatrix},\\
    E & = 
    \begin{bmatrix}
    E_{1} & \dotso & E_{n} & \dotso & E_{N}
    \end{bmatrix},
\end{align*}
where $D_{n},E_{n}\in\mbb{R}^{K}$ for each $n\in\mc{N}$. We now rewrite \eqref{eq:LinDistFlow} as 
\begin{equation*}
    \mbf{V}(t) = \wt{\mbf{V}}(t) - D\mc{P}(t) - E\mc{Q}(t).
\end{equation*}
Let us now concatenate over the time horizon \ab{$[0,\Delta T]$}, and write
\begin{equation*}
    \mbf{V} = 
    \begin{bmatrix}
    \mbf{V}(1)\\
    \vdots\\
    \mbf{V}(t)\\
    \vdots\\
    \mbf{V}(T)
    \end{bmatrix}\in\mbb{R}^{KT}
\end{equation*}
and similarly,
\begin{equation*}
    \wt{\mbf{V}} = 
    \begin{bmatrix}
    \wt{\mbf{V}}(1)\\
    \vdots\\
    \wt{\mbf{V}}(t)\\
    \vdots\\
    \wt{\mbf{V}}(T)
    \end{bmatrix}\in\mbb{R}^{KT}.
\end{equation*}
Letting
\begin{align*}
    \ol{D}_{n} & = \bigoplus_{1}^{T} D_{n}\in\mbb{R}^{KT\times T},\\ 
    \ol{E}_{n} & = \bigoplus_{1}^{T} E_{n}\in\mbb{R}^{KT\times T},
\end{align*}
we obtain
\begin{equation}\label{eq:power_flow}
    \mbf{V} = \wt{\mbf{V}} + \sum_{n=1}^{N}\left( \ol{D}_{n}\mbf{p}_{n} + \ol{E}_{n}\mbf{q}_{n}\right).
\end{equation}


\section{Centralized EV (dis)charging}\label{S:optimodel}

Here we formulate a centralized optimization problem to \nn{obtain} the optimal (dis)charge rates for each $EV_{n}$. First, at each node $i\in\mc{V}$ we constrain the voltage magnitude on each phase $\phi\in\Phi^{i}$ to stay within safe operational limits $[\ul{v}^{i:\phi},\ol{v}^{i:\phi}]$, where $\ul{v}^{i:\phi},\ol{v}^{i:\phi}\in\mbb{R}$ are the respective lower and upper bounds for operational limits.

Let us define the following vectors,
\begin{align*}
    \ul{\mbf{v}} & :=
    \begin{bmatrix}
        \begin{bmatrix}
            \set{
            \ul{v}^{1:\phi}
            }_{\phi\in\Phi^{1}}
        \end{bmatrix}^{\transpose}\\
    \vdots\\
        \begin{bmatrix}
            \set{
            \ul{v}^{i:\phi}
            }_{\phi\in\Phi^{i}}
        \end{bmatrix}^{\transpose}\\
    \vdots\\
        \begin{bmatrix}
            \set{
            \ul{v}^{k:\phi}
            }_{\phi\in\Phi^{k}}
        \end{bmatrix}^{\transpose}\\
    \end{bmatrix}\in\mbb{R}^{K},\\
    \ol{\mbf{v}}&:=
    \begin{bmatrix}
        \begin{bmatrix}
            \set{
            \ol{v}^{1:\phi}
            }_{\phi\in\Phi^{1}}
        \end{bmatrix}^{\transpose}\\
    \vdots\\
        \begin{bmatrix}
            \set{
            \ol{v}^{i:\phi}
            }_{\phi\in\Phi^{i}}
        \end{bmatrix}^{\transpose}\\
    \vdots\\
        \begin{bmatrix}
            \set{
            \ol{v}^{k:\phi}
            }_{\phi\in\Phi^{k}}
        \end{bmatrix}^{\transpose}\\
    \end{bmatrix}\in\mbb{R}^{K},
\end{align*}
and
\begin{equation*}
    \ul{\mbf{V}}:=
    \begin{bmatrix}
        \ul{\mbf{v}}\\
        \vdots\\
        \ul{\mbf{v}}
    \end{bmatrix}\in\mbb{R}^{KT},
    \quad
    \ol{\mbf{V}}:=
    \begin{bmatrix}
        \ol{\mbf{v}}\\
        \vdots\\
        \ol{\mbf{v}}
    \end{bmatrix}\in\mbb{R}^{KT}.
\end{equation*}
We now constrain the voltage as follows
\begin{equation}\label{eq:cons_volt}
    \ul{\mbf{V}}\leq \mbf{V} = \wt{\mbf{V}} + \sum_{n=1}^{N}\left( \ol{D}_{n}\mbf{p}_{n} + \ol{E}_{n}\mbf{q}_{n}\right) \leq \ol{\mbf{V}}.
\end{equation}
To make the presentation clearer, let us define
\begin{align*}
\Gamma_{n} & :=
\begin{bmatrix}
    \ol{D}_{n}\\
    -\ol{D}_{n}
\end{bmatrix}\in\mbb{R}^{2KT\times T},\\
\Xi_{n} & :=
\begin{bmatrix}
    \ol{E}_{n}\\
    -\ol{E}_{n}
\end{bmatrix}\in\mbb{R}^{2KT\times T},\\
\mbf{w}& :=
\begin{bmatrix}
    \ol{\mbf{V}} - \wt{\mbf{V}}\\
    -\ul{\mbf{V}} + \wt{\mbf{V}}
\end{bmatrix}\in\mbb{R}^{2KT}
\end{align*}
so we may now write \eqref{eq:cons_volt} as
\begin{equation*}
    \sum_{n=1}^{N}\Gamma_{n}\mbf{p}_{n}+\Xi_{n}\mbf{q}_{n}\leq \mbf{w}.
\end{equation*}

For each customer the operational cost over the time horizon \ab{$[0,\Delta T]$} is given by $\Omega_{n}(\mbf{p}_{n})$ \eqref{eq:cost_func}, and we wish to minimize the total operational cost \nn{of all customers in $\mc{N}$} ($\textstyle\sum_{n=1}^{N}\Omega_{n}(\mbf{p}_{n})$), subject to the feasibility sets $\mc{F}_{n}$ \eqref{eq:feasible_set}. We write the optimization problem \nn{\eqref{eq:EVmodel}} as
\begin{varsubequations}{$\mc{P}1$}
\label{eq:EVmodel}
\begin{align}
	\min_{\mbf{p}_{1},\dotsc,\mbf{p}_{n}\in\mbb{R}^{T}} & \ \sum_{n=1}^{N}\Omega_{n}(\mbf{p}_{n}) \label{eq:p1.1}\\
    \st & \ \left(\mbf{p}_{n},\mbf{q}_{n}\right)\in\mc{F}_{n} \label{eq:p1.2}\\
	\sts & \ \sum_{n=1}^{N} \Gamma_{n}\mbf{p}_{n} + \Xi_{n}\mbf{q}_{n} \leq \mbf{w}.\label{eq:p1.3}
\end{align}
\end{varsubequations}

\ab{
\begin{remark}
It is possible to include the power flow constraints (which are structurally similar to the inverter capacity constraint \eqref{eq:q_bound}) 
into \eqref{eq:EVmodel}. However, as our primary focus in this work is communication costs, and so for the sake of brevity we do not present this here.
\end{remark}
}

\section{Distributed Communication Censored Approach}\label{S:comms}

In this section we reformulate problem \eqref{eq:EVmodel} into an ADMM formulation. Then, we show that for this problem we can use the Communication-Censored-ADMM algorithm of \cite{liu2019communication}, with guarantees of convergence.

Let us begin with a quick summary of the general ADMM paradigm.

\subsection{ADMM Summary}\label{SS:admm}

Given an optimization problem
\begin{align}\label{eqadmm}
  \begin{split}
 \min_{\mathbf{x},\mathbf{z}} & \ f(\mathbf{x}) +g(\mathbf{z})  \\
 \st & \ \mathbf{A}\mathbf{x}+\mathbf{B}\mathbf{z}=\mathbf{c},\\
 &\ \mathbf{x}\in\mathbb{R}^n, \mathbf{z}\in \mathbb{R}^m.
    \end{split}
\end{align}
ADMM proceeds by first defining the augmented Lagrangian
\begin{align*}
L_{\alpha}(\mathbf{x},\mathbf{z},\mathbf{y}) & = f(\mathbf{x})+g(\mathbf{z})+\mathbf{y}^{T}(\mathbf{A}\mathbf{x}+\mathbf{B}\mathbf{z}-\mathbf{c}) \\
& \mathrel{\phantom{=}}+ \frac{\alpha}{2}\norm{\mathbf{A}\mathbf{x}+\mathbf{B}\mathbf{z}-\mathbf{c}}^{2}
\end{align*}
with dual variable $\mathbf{y}$, and regularization parameter $\alpha$ ($\alpha$ can also be thought of as a stepsize, which balances convergence with constraints satisfaction). The decision variables $\mbf{x,y,z}$ are then updated sequentially on each iteration $k$ as,
$$\mathbf{x}^{(k+1)} = \argmin_{\mathbf{x} \in\mathbb{R}^n} \  L_{\alpha}\left(\mathbf{x},\mathbf{z}^{(k)},\mathbf{y}^{(k)}\right), $$
$$\mathbf{z}^{(k+1)} = \argmin_{\mathbf{z} \in\mathbb{R}^m} \  L_{\alpha}\left(\mathbf{x}^{(k+1)},\mathbf{z},\mathbf{y}^{(k)}\right), $$
$$\mathbf{y}^{(k+1)}= \mathbf{y}^{(k)} - \alpha \left( \mathbf{A}\mathbf{x}^{(k+1)}+\mathbf{B}\mathbf{z}^{(k+1)}-\mathbf{c}\right).$$

The ADMM algorithm has many desirable properties, such as \emph{strong duality}, \emph{linear convergence rates}, etc. Importantly, the ADMM algorithm can be implemented in a distributed way, and can be applied to a variety of problems, as there are very few restrictions on the function $g$ in \eqref{eqadmm}. For a more detailed explanation, the reader should consult \cite{boyd2011distributed}.

\subsection{ADMM Formulation for \ref{eq:EVmodel}}

We consider a communication network represented by a graph $\mbb{G}:=\set{\mc{N},\mbb{E}}$, where the vertex set $\mc{N}$ is the set of \nn{EV} customers (as defined previously), and the edges \nn{in} $\mbb{E}\subseteq\mc{N}\times\mc{N}$ represent communication links between EV customers (an example is presented in Fig.~\ref{fig:my_label}). Specifically, an edge $(uv)\in\mbb{E}$ allows bidirectional communication between customer $u$ and customer $v$. Due to symmetry in the communication network, if edge $(uv)\in\mbb{E}$, then we must also have $(vu)\in\mbb{E}$. We denote the set of neighbors of customer $u$ as $\mc{N}_{u}:=\set{v\in\mc{N} \;\middle\vert\; (uv)\in\mbb{E}}$. We make the following assumption about $\mbb{G}$.

\begin{assumption}\label{ass1}
The graph $\mbb{G}$ is connected. 
\end{assumption}

\begin{figure}
    \centering
    \includegraphics[scale=0.45]{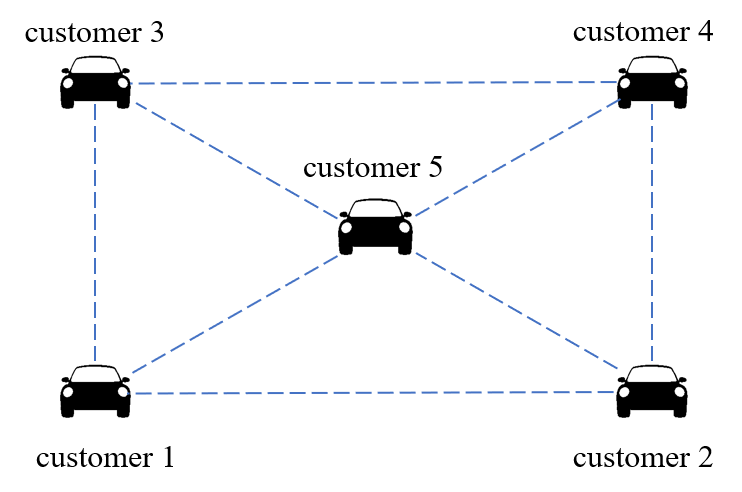}
    \caption{An example of a communication network represented by a graph $\mbb{G}:=\set{\mc{N},\mbb{E}}$}
    \label{fig:my_label}
\end{figure}

Let us introduce a slack variable $\mbf{s}_{n}\in\mbb{R}^{2KT}$ for each customer $n\in\mc{N}$, and define the following
\begin{align*}
    \Psi_{n} & :=
    \begin{bmatrix}
        \Gamma_{n} & \Xi_{n} & I
    \end{bmatrix}\in\mbb{R}^{2KT\times (2T+2KT)},\\
    \mbf{u}_{n} & :=
    \begin{bmatrix}
        \mbf{p}_{n}\\
        \mbf{q}_{n}\\
        \mbf{s}_{n}
    \end{bmatrix}\in\mbb{R}^{2T+2KT}.
\end{align*}
Then, we write the inequality constraint \eqref{eq:p1.3} as 
\begin{equation*}
    \sum_{n=1}^{N} \Psi_{n}\mbf{u}_{n} = \sum_{n=1}^{N}\Gamma_{n}\mbf{p}_{n} + \Xi_{n}\mbf{q}_{n} +\mbf{s}_{n} = \mbf{w}.
\end{equation*}
Define the sets 
\begin{align*}
\mc{S}_{n}&:=\set{\mbf{s}_{n}\in\mbb{R}^{2KT}\;\middle\vert\; \mbf{0}\leq\mbf{s}_{n}},\\
\mc{U}_{n}&:=\set{\mbf{u}_{n}\in\mbb{R}^{2T+2KT} \;\middle\vert\; \left(\mbf{p}_{n},\mbf{q}_{n}\right)\in\mc{F}_{n} \, ,\  \mbf{s}_{n}\in\mc{S}_{n}}.
\end{align*} 
Denote the indicator function by
\begin{equation*}
    \mc{I}_{n}(\mbf{u}_{n}) = 
    \begin{cases}
    0, & \text{ if } \mbf{u}_{n}\in\mc{U}_{n}\\
    \infty, & \text{ otherwise }
    \end{cases}
\end{equation*}
and let 
\begin{equation*}
    f_{n}(\mbf{u}_{n}) : =\Omega_{n}(\mbf{p}_{n}) + \mc{I}_{n}(\mbf{u}_{n}).
\end{equation*}

\begin{proposition}
The optimization problem \eqref{eq:EVmodel} is equivalent to the following
\begin{varsubequations}{$\mc{P}2$}
\label{eq:EVmodel_ADMM}
\begin{align}
    \min_{\mbf{u}_{1},\dotso,\mbf{u}_{N}\in\mbb{R}^{2T+2KT}} & \ \sum_{n=1}^{N}f_{n}(\mbf{u}_{n}) \label{eq:p2.1}\\
    \st & \ \sum_{n=1}^{N}\Psi_{n}\mbf{u}_{n} = \mbf{w}.\label{eq:p2.2}
\end{align}
\end{varsubequations}
\end{proposition}


Since strong duality holds for Problem \eqref{eq:EVmodel_ADMM}, together with Assumption \ref{ass1}, we can reformulate Problem \eqref{eq:EVmodel_ADMM} as a distributed consensus optimization problem using duality as in \cite{chang2014multi}. Specifically,

\begin{varsubequations}{$\mc{P}3$}
\label{eq:EVmodel_Dual}
\begin{align}
    \min_{ \substack{\bs{\lambda}_{i}\in\mbb{R}^{2KT}\\\beta_{ij}\in\mbb{R}^{2KT}}} & \ \sum_{n=1}^{N}\left(\varphi_{n}(\bs{\lambda}_{n})+\frac{1}{N}\bs{\lambda}_{n}^{\transpose}\mbf{w}\right) \label{eq:p3.1}\\
    \st & \ \bs{\lambda}_{n} = \beta_{nm} \ \forall m\in\mc{N}_{n}, \ n\in\mc{N} \label{eq:p3.2}\\
    \sts & \ \bs{\lambda}_{m} = \beta_{nm} \ \forall m\in\mc{N}_{n}, \ n\in\mc{N}. \label{eq:p3.2}
\end{align}
\end{varsubequations}

Here, $\bs{\lambda}_{n}$ is the $n^{\text{th}}$ customer\nn{'}s local copy of the global dual variable $\bs{\lambda}$, $\beta_{\nn{nm}}$ are auxiliary consensus variables, and for all $n\in\mc{N}$
$$
\varphi_{n}(\bs\lambda) = \max_{\mbf{u}_{n}\in\mbb{R}^{2T+2KT}}\left( -f_{n}(\mbf{u}_{n}) - \bs{\lambda}^{\transpose}\Psi_{n}\mbf{u}_{n} \right).
$$

Notice that Problem \eqref{eq:EVmodel_Dual} is in ADMM form \eqref{eqadmm}. Accordingly, we proceed to solve it using a 
distributed ADMM algorithm, which have been shown to perform well in \cite{makhdoumi2017convergence, shi2014linear}. 

The following assumption is necessary and standard for most optimization paradigms.

\begin{assumption}\label{ass2}
There exists an optimal solution $(\mbf{u}_{1}^{*},\dotsc,\mbf{u}_{N}^{*})$ to the problem \eqref{eq:EVmodel_ADMM}.
\end{assumption}

\begin{remark}
Note that as long as the feasibility sets $\mc{F}_{n}$ \eqref{eq:feasible_set} are non-empty, an optimal solution to \eqref{eq:EVmodel_ADMM} exists.
\end{remark}

\subsection{Communication-Censored-ADMM Formulation for \ref{eq:EVmodel_Dual}}



We start with a brief summary of the Communication-Censored-ADMM algorithm. 

At each iteration $k$, each $EV_{n}$ keeps the following local variables; the 
variable $\bs{\lambda}^{k}_{n}$, the dual variable $\nu^{k}_{n}$, the state variable $\wh{\bs{\lambda}}^{k}_{n}$ which was the last broadcast made to its neighbors, and the state variables of its neighbors $\wh{\bs{\lambda}}^{k}_{m}$ for $m\in\mc{N}_{n}$, which was the last broadcast from its neighbors. 

Let $\xi_{n}^{k} = \wh{\bs{\lambda}}_{n}^{k-1} - \bs{\lambda}_{n}^{k}$, and for constants $\gamma>0$, $\varepsilon\in(0,1)$ define
\begin{equation}
H_{n}(k,\xi_{n}^{k}) := \norm{\xi_{n}^{k}} - \gamma\varepsilon^{k}.
\end{equation}
Customer $n\in\mc{N}$ communicates with its neighbors only if $H_{n}(k,\xi_{n}^{k})\geq0$, and in this way communication is restricted to only significant updates of the local primal variables. The complete algorithm, which is executed by each EV $n\in\mc{N}$ in parallel, is presented below in \ref{alg1}.

\vspace{2pt}
\RestyleAlgo{boxruled}
\begin{algorithm}[h]
  \caption{}
  \label{alg1}
  \SetKwInOut{Input}{input}
  \SetKwInOut{Output}{output}
  \DontPrintSemicolon
  \Input{Initial local variables set to $\bs{\lambda}_{n}^{0} = \nu_{n}^{0} = \wh{\bs{\lambda}}_{n}^{0} = \mbf{0}$ and $\wh{\bs{\lambda}}_{m}^{0}=\mbf{0}$ for all $m\in\mc{N}_{n}$.\newline
  Iteration limit $S$. Stepsize $c$. \newline
  Local cost functions $\varphi_{n}$. \newline
  Censoring function $H_{n}$.}
  \For{iteration $k=1,\dotsc,S$ }{
  Compute local 
  variable $\bs{\lambda}_{n}^{k}$ as
  \begin{align*}
  \bs{\lambda}_{n}^{k} &= \argmin_{\bs{\lambda}_{n}}\Bigg( \varphi_{n}(\bs{\lambda}_{n}) +\frac{1}{N}\bs{\lambda}_{n}^{\transpose}\mbf{w}\\
  &+ \left\langle \bs{\lambda}_{n},\nu_{n}^{k-1} - c\sum_{m\in\mc{N}_{n}}(\wh{\bs{\lambda}}_{n}^{k-1} + \wh{\bs{\lambda}}_{m}^{k-1})  \right\rangle \\
  &+ c \abs{\mc{N}_{n}}\norm{\bs{\lambda}_{n}}^{2}\Bigg)
  \end{align*}\;
  Compute $\xi_{n}^{k} = \wh{\bs{\lambda}}_{n}^{k-1}-\bs{\lambda}_{n}^{k}$\;
  \eIf{$H_{n}(k,\xi_{n}^{k})\geq0$}{
  Transmit $\bs{\lambda}_{n}^{k}$ to neighbors and set $\wh{\bs{\lambda}}_{n}^{k} = \bs{\lambda}_{n}^{k}$
  }{
  Do not transmit and set $\wh{\bs{\lambda}}_{n}^{k} = \wh{\bs{\lambda}}_{n}^{k-1}$
  }
  \eIf{Recieve $\bs{\lambda}_{m}^{k}$ from any neighbor $m$}{
  Set $\wh{\bs{\lambda}}^{k}_{m} = \bs{\lambda}_{m}^{k}$
  }{
  Set $\wh{\bs{\lambda}}_{m}^{k} = \wh{\bs{\lambda}}_{m}^{k-1}$
  }
  Update the local dual variable as
  $$
  \nu_{n}^{k} = \nu_{n}^{k-1} + c\sum_{m\in\mc{N}_{n}}(\wh{\bs{\lambda}}_{n}^{k} - \wh{\bs{\lambda}}_{m}^{k})
  $$
  }
\end{algorithm}

\begin{proposition}
For problem \eqref{eq:EVmodel_Dual}, \ref{alg1} converges to an optimal solution $(\bs{\lambda}_{1}^{*},\dotsc,\bs{\lambda}_{N}^{*})$.
\end{proposition}

\begin{proof}
It can be easily shown that the local cost functions 
$$
\varphi_{n}(\bs{\lambda}_{n})+\frac{1}{N}\bs{\lambda}_{n}^{\transpose}\mbf{w}
$$
are convex. 
Taking convexity of the cost functions together with Assumption \ref{ass1} and Assumption \ref{ass2}, the result now follow\nn{s} Theorem 1 of \cite{liu2019communication}. 
\end{proof}

Note that the \nn{$\bs{\lambda}_{n}$} update step in \ref{alg1} involves solving the following min-max optimization problem,
\begin{align*}
\bs{\lambda}_{n}^{k} = \argmin_{\bs{\lambda}_{n}}\max_{\mbf{u}_{n}}
\Bigg\{
-f_{n}(\mbf{u}_{n}) - \bs{\lambda}_{n}^{\transpose}\Psi_{n}\mbf{u}_{n} + \frac{1}{N}\bs{\lambda}_{n}^{\transpose}\mbf{w} \\
+ \left\langle \bs{\lambda}_{n},\nu_{n}^{k-1} - c\sum_{m\in\mc{N}_{n}}(\wh{\bs{\lambda}}_{n}^{k-1} + \wh{\bs{\lambda}}_{m}^{k-1})  \right\rangle\\
+  c\abs{\mc{N}_{n}}\norm{\bs{\lambda}_{n}}^{2} \Bigg\},
\end{align*}
where $c$ is the stepsize (cf. $\alpha$ in Section \ref{SS:admm}).

Following \cite[Section IV-A]{chang2014multi}, as the objective of this min-max problem is convex in $\bs{\lambda}_{n}$ and concave in $\mbf{u}_{n}$, we can use the Minimax Theorem of \cite[Proposition 2.6.2]{bertsekas2003convex} to switch the order of the $\max$ and $\min$ operators. Let 
$$
\mbf{y} = \bs{\Psi_{n}}\mbf{u}_{n} - \frac{1}{N}\mbf{w} - \nu_{n}^{k-1} + c\sum_{m\in\mc{N}_{n}}(\wh{\bs{\lambda}}_{n}^{k-1} + \wh{\bs{\lambda}}_{m}^{k-1}).
$$
Then the inner minimization problem has a solution
\begin{equation}\label{eq:lambda_sol}
\begin{aligned}
2c\abs{\mc{N}_{n}}\bs{\lambda}_{n}^{k} = & \mbf{y} 
\end{aligned}
\end{equation}
and the outer maximization problem reduces to
\begin{equation}\label{eq:lambda_sol2}
\mbf{u}_{n}^{k} = \argmin_{\mbf{u}} \set{ f_{n}(\mbf{u}) + \frac{1}{4c\abs{\mc{N}}}\norm{\mbf{y}
}^{2} }.
\end{equation}
Updating the variable $\bs{\lambda}_{n}^{k}$, by first updating $\mbf{u}_{n}^{k}$ \nn{as per \eqref{eq:lambda_sol2}} and then using \eqref{eq:lambda_sol} is computationally efficient. That is, \ref{alg1} can be run in parallel by each EV, facilitating efficient paralleled EV charge and discharge operations. 

\ab{
\begin{remark}\label{comms-robust}
In the CC-ADMM paradigm, robustness to a communication failure of an individual agent is inherited from the underlying peer-to-peer communication network. That is, CC-ADMM is designed to limit the communication between agents at particular times based on the relevance of the transmitted information --- it does not remove the \emph{ability} to communicate. 
\end{remark}
}

\section{Numerical Simulations}\label{S:numerics}

To demonstrate \ref{alg1} for EV (dis)charging, we consider a representative unbalanced 
three-phase, two node distribution circuit with time-varying loads across a day \ab{($T=48, \Delta = 0.5$)}. The circuit has unbalanced impedences 
--- for instance, $Z^{01,\mf{a}\mf{a}} = 0.1313 + 0.3856\mf{i}$, $Z^{01,\mf{b}\mf{b}} = 0.1278 + 0.3969\mf{i}$, $Z^{01,\mf{c}\mf{c}} = 0.1293 + 0.3920\mf{i}$ ($\mf{i}^{2}=-1$) 
\footnote{The complete data of impedences and loads, and the simulation settings for all our examples is openly available at \url{https://github.com/Abhishek-B/IREP-2022}.}. In what follows, there are 50 customers connected to each of the three supply points, as per the topology illustrated in Fig. \ref{fig:Example_1}. Note that each customer connects to the distribution circuit supply points via a PCC, as per the residential EV energy system depicted in Fig.~\ref{fig:convention}.
\begin{figure}
    \centering
    \includegraphics[height=6cm]{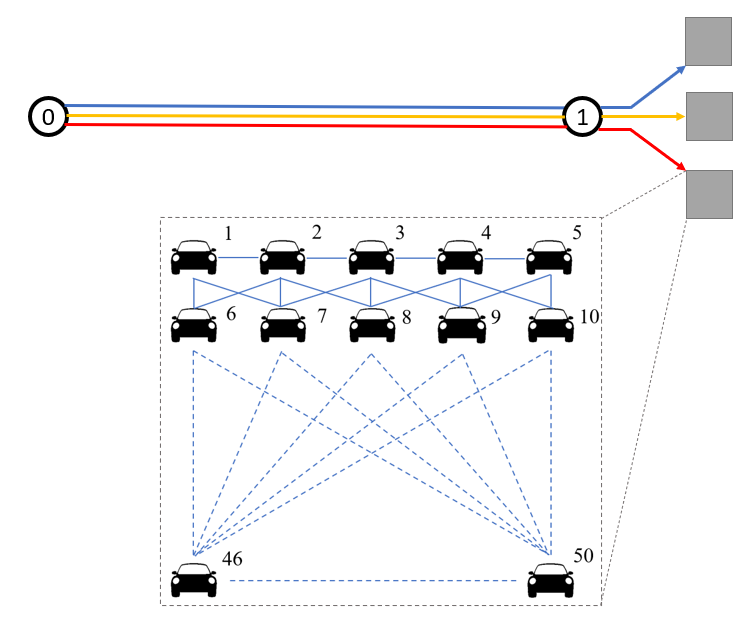}
    \caption{Topology for a representative unbalanced three-phase, two node distribution circuit with $50$ EV customers connected to each single phase supply point (represented by grey boxes at node 1). The PCC for each customer as in Fig.~\ref{fig:convention}, is located at one of these three supply points.
    }
    \label{fig:Example_1}
\end{figure}

We set the following\nn{:} $a_{n}=0,\quad d_{n}=T,\quad \ol{s}_{n}=12,\quad \ol{p}_{n} = 7, \quad \ul{p}_{n} = -7,$ and the operational cost regularization $\kappa_{n}=10^{-4}$, for all EVs. The remaining parameters $b_{n}, \hat{\sigma}_{n}, \sigma^{*}_{n}, \ol{\sigma}_{n}, \ul{\sigma}_{n}$ are chosen randomly in accordance with the simulation setup of \cite{nimalsiri2021coordinated}. We do not list these parameters here, but all of our code and data is publicily available at \url{https://github.com/Abhishek-B/IREP-2022}, so that the reader can access our simulation setup, and test parameters there.

We compare \ref{alg1} against the benchmark ADMM algorithm of \cite{nimalsiri2021distributed} (by turning off the censoring), with step size $c=100$, and iteration limit $S=30$ for both algorithms. 

\begin{example}\label{example1}
We consider a completely connected 
communication network
, i.e., every customer can communicate with every other customer (for $u\in\mc{N}, \mc{N}_{u} = \mc{N}\backslash\{u\}$). Our simulation results are illustrated in Fig.~\ref{fig:result} and Fig.~\ref{fig:censoring}.

\begin{figure*}[]
    \centering
    \includegraphics[width=1\textwidth]{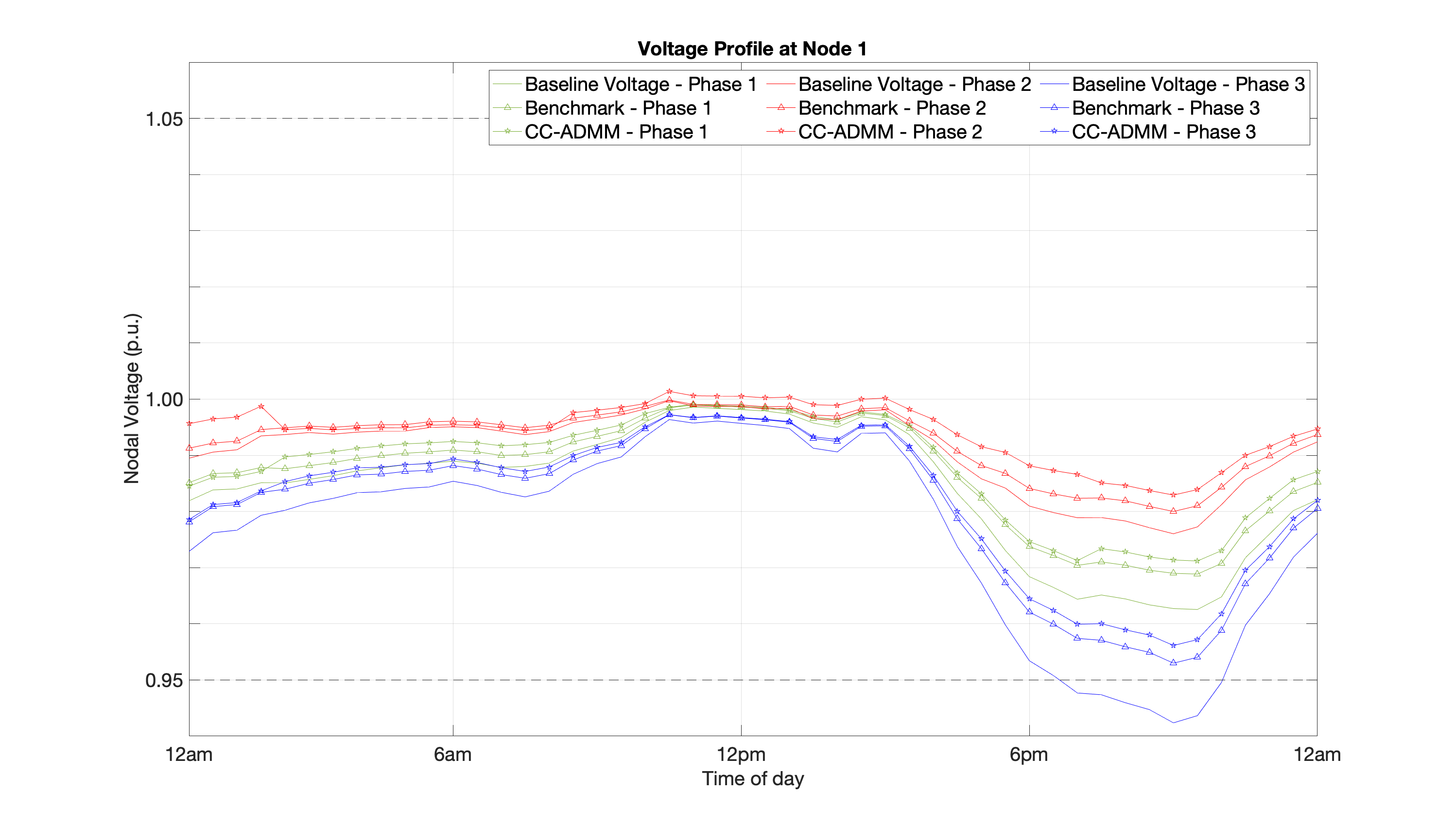}
    \caption{Voltage profile at node~1 for each of the three phases in Example 1. The baseline voltage profile for each phase represents the distribution circuit without grid-connected EV loads --- and we observe a voltage excursion between 6pm and midnight occurring on phase 3. We benchmark the ADMM algorithm from \cite{nimalsiri2021distributed} against \ref{alg1}, and observe a flatter voltage profile in all cases --- absent voltage excursions. Both ADMM algorithms correct the observed voltage excursion on Phase 3 by way of EV discharging.}
    \label{fig:result}
\end{figure*}

Firstly, notice from Fig. \ref{fig:result} that the solution obtained from the benchmark algorithm of \cite{nimalsiri2021distributed} and the censored solution of \ref{alg1} both satisfy safe operational limits of $\pm5\%$ about the nominal voltage. Moreover, the censored solution of \ref{alg1} presents with a generally flatter voltage profile. 

The real success comes from looking at the censoring pattern from \ref{alg1}. In Fig.~\ref{fig:censoring}, a white square means that a customer broadcast information at that iteration, and black square means they did not. In Fig.~\ref{fig:censoring}, we observe that the communication is very sparse. In fact, \ref{alg1} required $21\%$ of the communications that would be otherwise by required by the benchmark algorithm of \cite{nimalsiri2021distributed}. Our simulation results are consistent with findings from \cite{tsianos2012communication}, where it is discovered that communicating less frequently, provides improvements to the optimization methods over networks with communication. 

\begin{figure}[]
    \centering
    \includegraphics[width=8cm]{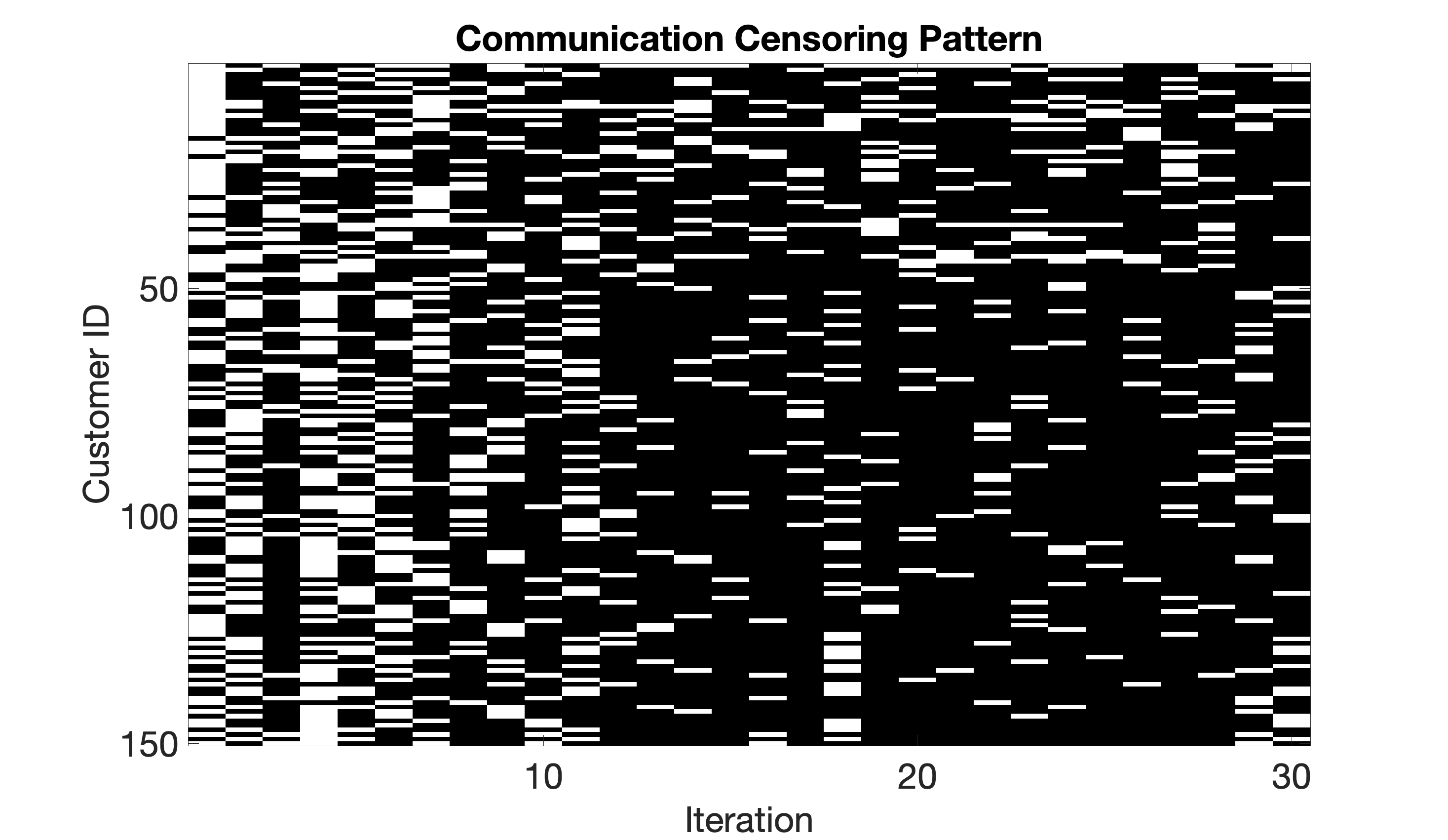}
    \caption{Peer-to-peer customer communications at each iteration as per CC-ADMM for Example \nn{\ref{example1}}. Specifically, a white square indicates that a customer has broadcast information for that iteration, and a black square indicates that the customer did not communicate during the respective iteration.}
    \label{fig:censoring}
\end{figure}

\end{example}

\begin{example}\label{example2}

Next, we consider a communication network where each customer communicates with only 70 other grid-connected customers.

\begin{figure*}[ht!]
    \centering
    \includegraphics[width=1\textwidth]{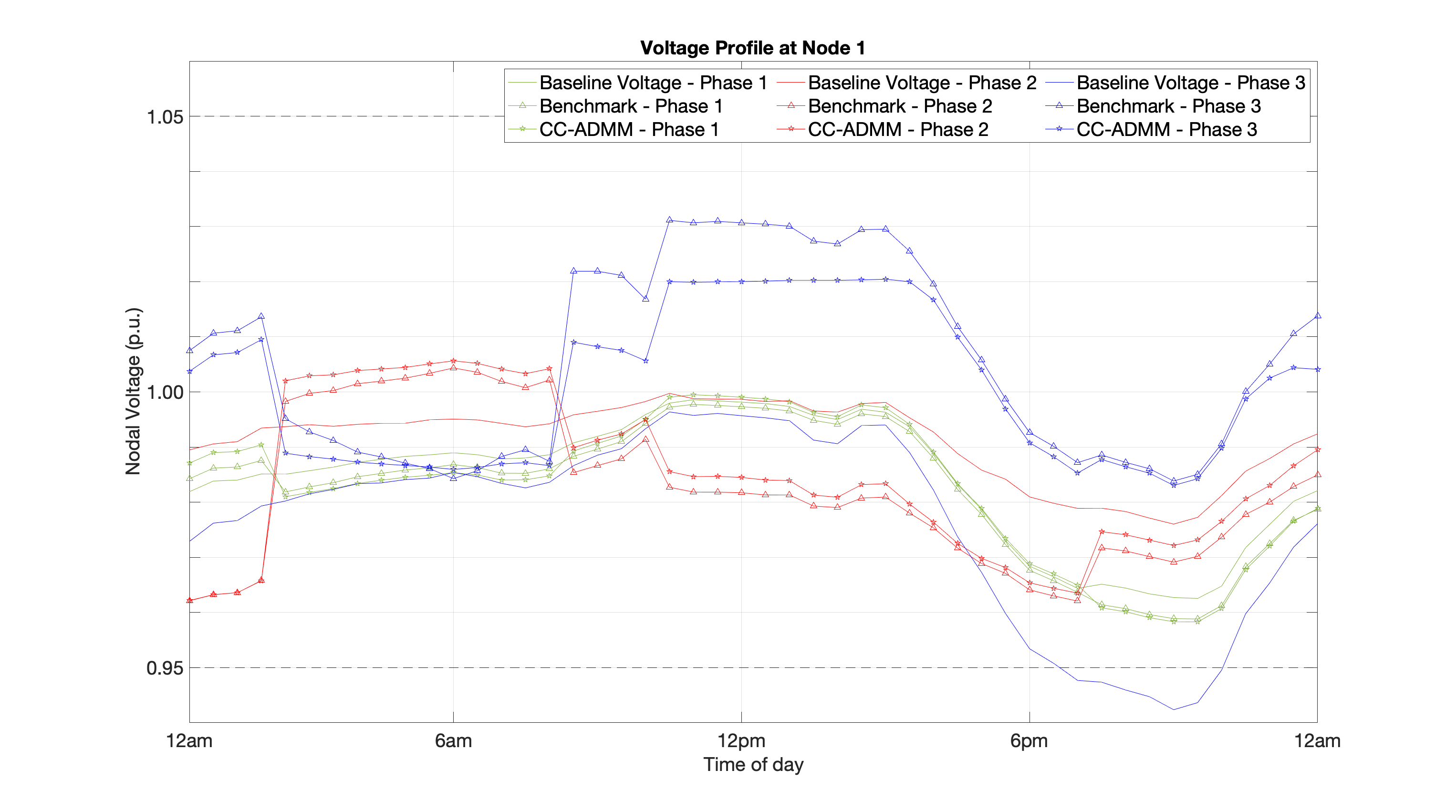}
    \caption{Voltage profile at node~1 for each of the three phases in Example 2. 
    We benchmark the ADMM algorithm from \cite{nimalsiri2021distributed} against \ref{alg1}, and observe a voltage profile which is generally closer to the nominal voltage with \ref{alg1} in all cases --- absent voltage excursions. Both ADMM algorithms correct the observed voltage excursion on Phase 3 by way of EV discharging.}
    \label{fig:result2}
\end{figure*}

Again, we observe from Fig.~\ref{fig:result2} that the solution obtained by the benchmark algorithm of \cite{nimalsiri2021distributed} and \ref{alg1} are qualitatively the same, and both satisfy the safe operational constraints of $\pm5\%$ about the nominal voltage. In fact, the solution \nn{from} \ref{alg1} generally \nn{results in a} voltage profile closer to the nominal voltage ($1$~p.u.). 

Overall, \ref{alg1} required $17\%$ of the communication that was otherwise required by the benchmark algorithm from \cite{nimalsiri2021distributed}. From Fig. \ref{fig:censoring2}, we observe that the  communication censoring pattern is sparser in the earlier iterations in comparison to Example~\ref{example1}. However, we also observe that more customers communicate during later iterations when compared with Example \ref{example1}. Accordingly, we observe that \ref{alg1} reduces the transmission of unnecessary information between EV customers, in the context of a peer-to-peer communication network.

\begin{figure}[h]
    \centering
    \includegraphics[width=8cm]{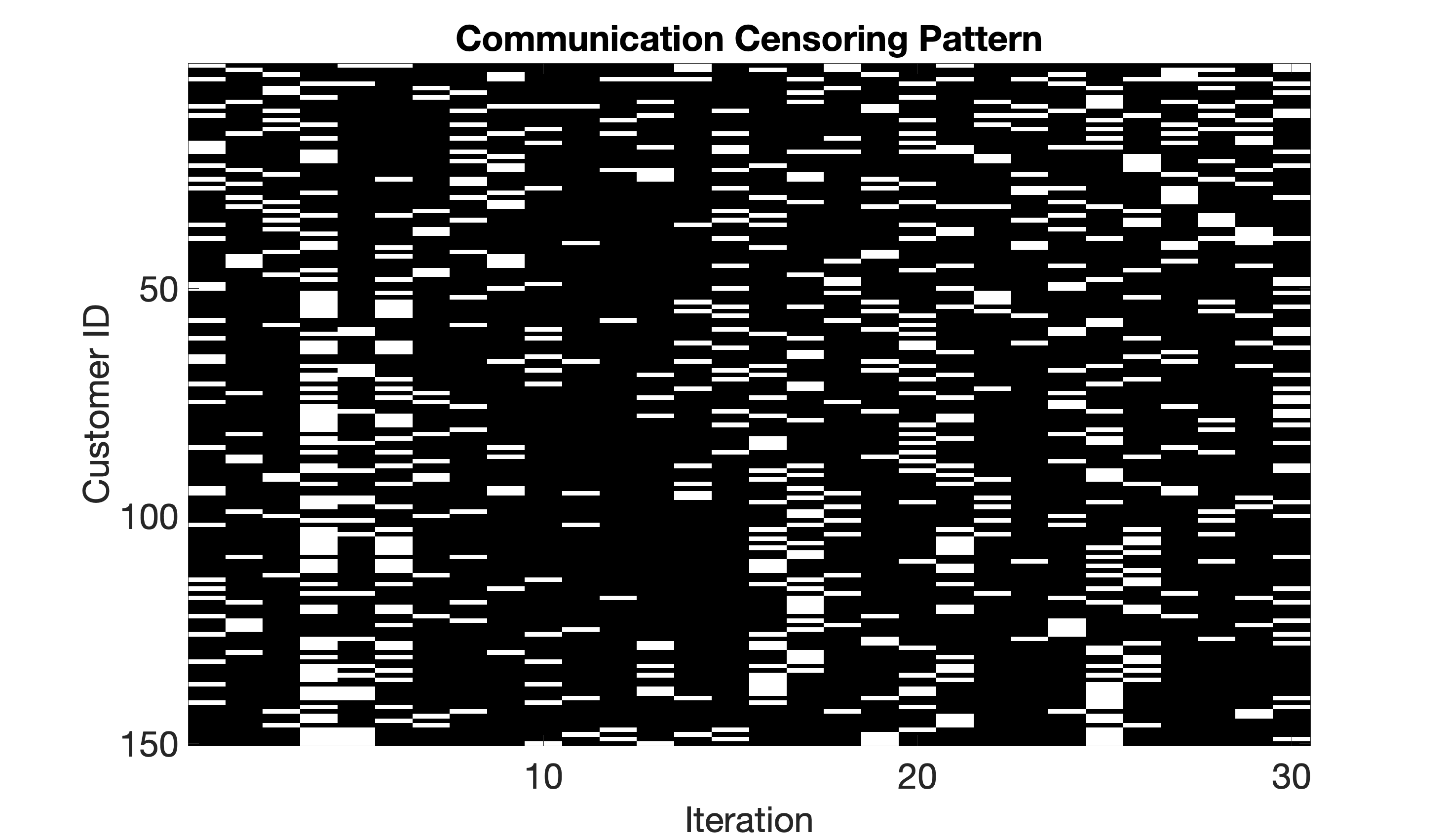}
    \caption{Peer-to-peer customer communications at each iteration as per CC-ADMM for Example 2. Specifically, a white square indicates that a customer has broadcast information for that iteration, and a black square indicates that the customer did not communicate during the respective iteration.}
    \label{fig:censoring2}
\end{figure}

\end{example}

The communication savings observed in these examples are remarkable, and it speaks to the potential of the censoring strategy for EV (dis)charge coordination in larger scale systems with more complicated dynamics. 

\section{Conclusion \& Future Work}

In this paper, we proposed an approach to coordinating EV (dis)charging in unbalanced electrical networks, using a Communication-Censored-ADMM algorithm. We proved that the algorithm is guaranteed to converge to the optimal solution, in the EV (dis)charge setting. The censoring strategy improved the optimization-based EV (dis)charging operations over an electrical network, while also reducing peer-to-peer communications. 
In the presented case study, we demonstrated the potential benefits of censoring communication between EVs. In future work, more realistic unbalanced electrical networks can be considered for benchmarking the proposed Communication-Censored-ADMM EV (dis)charging against other ADMM-based approaches. \ab{Ways to incorporate communication \nnrev{status (success/failure)} checks, such as \emph{beaconing} \cite[Section 2.1]{badonnel2006fault}, in line with Remark \ref{comms-robust}, are also possible. 
}


\section*{Acknowledgment}

The first author would like to thank Dr. P. Braun and Prof. I. Shames for insightful discussions throughout the project.



%

\bibliographystyle{IEEEtran}
\bibliography{ref.bib}




\end{document}